\documentclass[11pt]{article}

\usepackage[a4paper,margin=1in]{geometry}
\usepackage[T1]{fontenc}
\usepackage{lmodern}
\usepackage{amsmath,amssymb,amsthm,mathtools}
\usepackage{booktabs}
\usepackage{url,authblk,cite}
\usepackage{microtype}
\usepackage{enumitem}
\usepackage{xcolor}
\usepackage{tikz}

\setlist{nosep}

\newtheorem{theorem}{Theorem}[section]
\newtheorem{lemma}[theorem]{Lemma}
\newtheorem{proposition}[theorem]{Proposition}
\newtheorem{corollary}[theorem]{Corollary}

\theoremstyle{remark}
\newtheorem{remark}[theorem]{Remark}

\newcommand{\R}{\mathbb R}
\newcommand{\ee}{\mathbf e}
\newcommand{\Dr}{D_{\!r}}

\newcommand{\cS}{\mathcal S}
\newcommand{\cL}{\mathcal L}
\newcommand{\floor}[1]{\left\lfloor #1\right\rfloor}

\usepackage{hyperref}
\hypersetup{
  hidelinks,
  pdftitle={A proof of the resistance diameter conjecture for line graphs},
  pdfauthor={Xiang-Feng Pan, Xiang-Yang Liu, Zhen-Mu Hong},
  pdfkeywords={Effective resistance, resistance diameter, line graph}
}

\title{A proof of the resistance diameter conjecture for line graphs}

\author[a]{Xiang-Feng Pan}
\author[a]{Xiang-Yang Liu}
\author[b]{Zhen-Mu Hong
{\thanks{Corresponding author. \protect \\
\indent \ \
E-mail:
    xfpan@ahu.edu.cn (X.-F. Pan),
    2484262246@qq.com  (X.-Y. Liu), zmhong@mail.ustc.edu.cn (Z.-M. Hong) }}}

\affil[a]{\small School of Mathematical Sciences, Anhui University, Hefei 230601, China}
\affil[b]{\small School of Statistics and Applied Mathematics, Anhui University of Finance \& Economics, Bengbu 233030, China}

\date{}

\begin{document}
\maketitle

\begin{abstract}
We prove the conjecture of Xu, Li, Hua, and Pan that the resistance diameter
does not increase under the line-graph operation. For every finite connected
simple graph $G$ with at least one edge, we establish
$\Dr(L(G))\le\Dr(G)$, with equality if and only if $G$ is a cycle or $K_4$.
The proof combines an exact electrical representation of $L(G)$ by stars
with a sharp budget inequality for the branch core. This inequality controls
the combined error terms arising from the comparison of degree-two paths
and is central to both the diameter bound and the equality analysis.
\end{abstract}



\section{Introduction}
\subsection{Background and related work}

Unless stated otherwise, graphs are finite, simple, and undirected; auxiliary electrical networks may have parallel edges.  The line graph $L(G)$ of a graph $G$ has vertex set $E(G)$, with two vertices adjacent whenever the corresponding edges of $G$ have a common endpoint.  Thus the line-graph operation replaces each local star of $G$ by a clique.  Unlike
ordinary graph distance, however, effective resistance depends on all routes between two vertices, so its behavior under this operation is not determined by shortest paths alone.

Resistance distance was introduced systematically by Klein and Randi\'c \cite{KleinRandic1993}.  If every edge of a connected graph is replaced by a unit resistor, the resistance distance $R_G(x,y)$ is the effective resistance measured between $x$ and $y$.  It admits equivalent formulations through electrical flows, the Moore--Penrose inverse of the graph Laplacian, and random walks; standard references include \cite{Bapat2014,DoyleSnell1984,LyonsPeres2017}.  The Kirchhoff index is the sum of the resistance distances over all unordered vertex pairs \cite{GutmanMohar1996}.  Here we study instead the maximum pairwise resistance
\[
   \Dr(G)=\max_{x,y\in V(G)}R_G(x,y).
\]
We allow $x=y$ in the maximum; in particular, the one-vertex graph has resistance diameter $0$.  This quantity is called the \emph{resistance diameter}.

The systematic study of resistance diameter is comparatively recent.
Sardar, Hua, Pan, and Raza obtained a combinatorial formula for hypercubes
\cite{SHPR2020}, and Li, Xu, Hua, and Pan treated Cartesian and
lexicographic products of paths \cite{LXHP2022}.  Evans and Francis
survey exact and algorithmic methods for resistance distances on structured
graphs \cite{EvansFrancis2022}.  On the methodological side, Thomson's
principle and Rayleigh monotonicity connect effective resistance to minimum
energy flows \cite{DoyleSnell1984,LyonsPeres2017}; Laplacian and optimization
viewpoints are developed in \cite{Devriendt2022,GBS2008}; and
Kron reduction provides the Schur-complement formalism for eliminating
interior vertices while preserving boundary resistances
\cite{DorflerBullo2013}.

Xu, Li, Hua, and Pan \cite{XLHP22} introduced the problem considered
here.  Using series--parallel reductions,
elimination, and star--mesh transformations, they proved
\[
\Dr(L(G))\le \Dr(G)
\]
for trees and unicyclic graphs. They also reported computational evidence for all simple nonempty connected graphs with fewer than twelve vertices and posed the general statement as their Conjecture~5.1.  Beyond the unicyclic
case, several cyclic routes interact within the branch core.  A
transformation that is sharp on one thread need not remain sharp when two selected threads are coupled through this core.

\subsection{Main result and proof outline}

Our main theorem resolves the conjecture of Xu et al. \cite{XLHP22}.

\begin{theorem}\label{thm:main}
For every finite connected simple graph $G$ with at least one edge,
\[
   \Dr(L(G))\le \Dr(G),
\]
with equality if and only if
$G\cong C_n$ $(n\ge 3)$ or $G\cong K_4$.
\end{theorem}

Thus the constant $1$ cannot be improved, and the theorem determines exactly when equality occurs.  The proof is based on electrical network theory, but Rayleigh monotonicity alone does not compare the two diameters, because the locations of the terminals change when a star is replaced by a clique.

The starting point is the classical star--mesh transformation: a star
whose $d$ arms have resistance $1/d$ is equivalent, at its boundary, to a
unit-resistance clique. Applying this at every vertex realizes $L(G)$ as
an incidence network. A maximal degree-two path of length $L$, with endpoint
degrees $p,q$, then has compressed resistance
\[
 \tau=L-1+\frac1p+\frac1q.
\]
The difficulty is that the line-graph vertices lie inside these compressed
paths, whereas the diameter of $G$ is measured at the original vertices.

We separate this terminal-location problem from the graph structure.
Proposition~\ref{prop:general-transfer} gives a comparison for arbitrary
weighted networks and subdivision grids: independent rounding of two
interior points produces explicit errors $C_h,C_k$. The graph-specific
step is the full-core budget, Lemma~\ref{lem:CB}, which proves
$C_h+C_k\le(1-\beta)D$ for the compression determined by the full degrees
of $G$. The two errors must be estimated together because they share the
same electrical core. Lemma~\ref{lem:budget-strict} also identifies a whole
parameter range in which this budget is strict.

A second general tool, Lemma~\ref{lem:rank-one}, concerns the shortening of
one resistor in an arbitrary weighted network. For any probability demand,
one of the two adjacent original grid points controls the selected point
on the shortened resistor. This estimate handles a fixed root and hence
the loop-thread configurations. Thus the network comparison and
interpolation tools have a scope beyond line graphs; the budget supplies
the degree-dependent information needed here.

Section~2 develops the network tools. Section~3 reduces the line-graph
problem to its branch core, Section~4 proves the budget, and Section~5
handles loop threads. Section~6 completes the diameter bound and determines
equality, and Section~7 records further consequences. The appendices contain
the local algebra and the proof of the interpolation lemma.

\section{Electrical comparison tools}\label{sec:tools}

\subsection{Energy, reduction, and wiring}
An auxiliary network is a finite connected loopless multigraph with
positive edge resistances; parallel edges are allowed. Write $\cL_N$ for
its weighted Laplacian, $\cL_N^+$ for the pseudoinverse, and
$\R_0^{V(N)}=\{\zeta:\mathbf1^{\mathsf T}\zeta=0\}$. On this space put
\begin{equation}\label{eq:energy-form}
 \langle\zeta,\eta\rangle_N=\zeta^{\mathsf T}\cL_N^+\eta,
 \qquad Q_N(\zeta)=\|\zeta\|_N^2=\zeta^{\mathsf T}\cL_N^+\zeta.
\end{equation}
If $\partial_N$ records the net current leaving each vertex, Thomson's
principle and the resistance formula are
\begin{equation}\label{eq:thomson}
 Q_N(\zeta)=\min_{\partial_N\theta=\zeta}\sum_e r_e\theta_e^2,
 \qquad R_N(x,y)=Q_N(\ee_x-\ee_y).
\end{equation}
A unit-mass vector has coordinate sum one; a probability vector also has
nonnegative coordinates.

We use the standard electrical principles in
\cite{DoyleSnell1984,LyonsPeres2017,KleinRandic1993,Strutt1871}.
Resistance is a metric. Series resistances add, and parallel conductances
add. A subnetwork attached at one vertex carries no current between
terminals outside its interior and may be deleted. In particular,
resistances add across a cut vertex. Decreasing resistances, adding edges,
or identifying vertices cannot increase effective resistance (Rayleigh
monotonicity); consequently $R_G(x,y)\le R_H(x,y)$ for a connected subgraph
$H$ containing $x,y$. If a shortened edge carries nonzero current in the
old electrical flow, the decrease is strict: that flow has strictly smaller
energy in the new network. Finally, Schur elimination preserves the
boundary response matrix and all resistances between retained vertices
\cite{DorflerBullo2013}.

The two inequalities in the next lemma are due to Coppersmith, Feige, and
Shearer~\cite{CFS1996}.  We include a self-contained
proof and track equality in part~\textup{(ii)}, which is used in
Section~\ref{sec:equality}.
\begin{lemma}[Degree lower bounds]\label{lemma:degree-bound}
Let $G$ be a connected simple graph and let $x,y\in V(G)$ be distinct.
\begin{enumerate}[label=\textup{(\roman*)}]
\item If $x$ and $y$ are nonadjacent, then
\[
 R_G(x,y)\ge \frac{1}{d_G(x)}+\frac{1}{d_G(y)}.
\]
\item If $x$ and $y$ are adjacent, then
\[
 R_G(x,y)\ge \frac{1}{d_G(x)+1}+\frac{1}{d_G(y)+1}.
\]
Equality in the second inequality holds if and only if $N_G[x]=N_G[y]$.
\end{enumerate}
\end{lemma}
\begin{proof}
Let $\theta$ be the electrical unit $x$--$y$ flow.  If $x$ and $y$ are
nonadjacent, the two sets of incident edges are disjoint.  Cauchy--Schwarz
at the source and sink gives
\[
 \sum_{e\ni x}\theta_e^2\ge\frac1{d_G(x)},\qquad
 \sum_{e\ni y}\theta_e^2\ge\frac1{d_G(y)}.
\]
Their sum is part of the total energy, proving~\textup{(i)}.

Now suppose $xy\in E(G)$ and put $p=d_G(x)$, $q=d_G(y)$.  If $p=1$ or
$q=1$, the edge $xy$ is a bridge at a leaf and the assertion is immediate;
equality occurs only for $G=K_2$, when the closed neighborhoods agree.
Assume $p,q\ge2$, orient $xy$ from $x$ to $y$, and write
$c=\theta_{xy}$.  The remaining edges at $x$ carry total current $1-c$ away
from $x$, and the remaining edges at $y$ carry total current $1-c$ into
$y$.  Another application of Cauchy--Schwarz yields
\begin{align*}
 R_G(x,y)=\sum_e\theta_e^2
 &\ge c^2+\frac{(1-c)^2}{p-1}+\frac{(1-c)^2}{q-1}\\
 &\ge \frac{p+q-2}{pq-1}
 =\frac1{p+1}+\frac1{q+1}
   +\frac{(p-q)^2}{(p+1)(q+1)(pq-1)}.
\end{align*}
This proves~\textup{(ii)}.  Equality in its stated bound first forces
$p=q=:d$, equality in both Cauchy--Schwarz estimates, and zero current on
every edge not incident with $x$ or $y$.  With the sink potential normalized
to $0$, the source potential is $2/(d+1)$ and every other neighbor of either
terminal has potential $1/(d+1)$.  Kirchhoff's law then forces each neighbor
of $x$ other than $y$ also to be adjacent to $y$, and conversely; hence
$N_G[x]=N_G[y]$.

Conversely, if $N_G[x]=N_G[y]$, write
$d=d_G(x)=d_G(y)$ and assign potentials $1$ at $x$, $0$ at $y$, and $1/2$
at every other vertex.  This is harmonic away from $x,y$ and the source
current is $(d+1)/2$, so $R_G(x,y)=2/(d+1)$, which is equality in the stated
bound.
\end{proof}

For a nonempty vertex set $S$ and $z\notin S$, let $N/S$ identify $S$ to
one vertex $s_*$, and write $\mathcal P(S)$ for the probability vectors
supported on $S$.

\begin{lemma}[Wiring and variance]\label{lem:set-contraction}
Define $R_N(z,S)=R_{N/S}(z,s_*)$. Then
\begin{equation}\label{eq:set-contraction}
 R_N(z,S)=\min_{\mu\in\mathcal P(S)}Q_N(\ee_z-\mu).
\end{equation}
For every $\nu\in\mathcal P(S)$,
\begin{equation}\label{eq:variance-bound}
 \max_{x\in\operatorname{supp}\nu}R_N(z,x)
 \ge R_N(z,S)+\frac12\sum_{x,y}\nu_x\nu_yR_N(x,y).
\end{equation}
In particular, if $x,y\in S$ and $R_N(x,y)=1/2$, then
\begin{equation}\label{eq:half-resistance}
 \max\{R_N(z,x),R_N(z,y)\}\ge R_N(z,S)+\frac18.
\end{equation}
\end{lemma}
\begin{proof}
A flow with demand $\ee_z-\mu$ descends after contraction to a unit
$z$--$s_*$ flow with no greater energy. Conversely, pull back the electrical
potential from $N/S$, setting it to zero on $S$. The maximum principle
makes every extraction at $S$ nonnegative, so these extractions form an
admissible $\mu$. Thomson's principle proves \eqref{eq:set-contraction}.
Since $\nu$ is a probability vector, expanding $\|\ee_z-\nu\|_N^2$ as a
convex combination of two-point energies,
$\|\ee_z-\nu\|_N^2=\sum_x\nu_x\|\ee_z-\ee_x\|_N^2
-\frac12\sum_{x,y}\nu_x\nu_y\|\ee_x-\ee_y\|_N^2$, and using
$R_N(i,j)=\|\ee_i-\ee_j\|_N^2$, gives
\begin{equation}\label{eq:set-variance-identity}
 Q_N(\ee_z-\nu)=\sum_x\nu_xR_N(z,x)
 -\frac12\sum_{x,y}\nu_x\nu_yR_N(x,y).
\end{equation}
The left side is at least $R_N(z,S)$, and the first term on the right is at
most the stated maximum. The last assertion uses
$\nu=(\ee_x+\ee_y)/2$.
\end{proof}

\subsection{Point insertion and comparison under compression}

\begin{lemma}[Point insertion]\label{lem:point-insertion}
Let $0<t<1$, and split a resistor $uv$ of resistance $r$ in a network $N$ at a new point $z$, assigning resistance $rt$ to $uz$ and $r(1-t)$ to $zv$.  Denote the subdivided network by $\widetilde N$.  For every unit-mass
vector $\eta$ supported on $V(N)$,
\begin{equation}\label{eq:point-insertion}
 Q_{\widetilde N}(\ee_z-\eta)
 =Q_N\bigl((1-t)\ee_u+t\ee_v-\eta\bigr)+rt(1-t).
\end{equation}
\end{lemma}

\begin{proof}
Orient $uv$ from $u$ to $v$, and put $\nu=(1-t)\ee_u+t\ee_v$.
Given a flow in $N$ with demand $\nu-\eta$ and current $j$ on $uv$, keep
all other currents and assign currents $j-(1-t)$ and $j+t$ to $u z$ and
$zv$, respectively. The resulting flow has demand $\ee_z-\eta$ in
$\widetilde N$. This correspondence is bijective, and
\[
 rt\,[j-(1-t)]^2+r(1-t)(j+t)^2=rj^2+rt(1-t).
\]
Thus the energies of corresponding flows differ by the constant
$rt(1-t)$. Taking minima in Thomson's principle proves
\eqref{eq:point-insertion}.
\end{proof}

Taking $\eta=\ee_x$ and expanding the energy also gives
\begin{equation}\label{eq:degree-two-specialized}
 R_{\widetilde N}(z,x)
 =(1-t)R_N(u,x)+tR_N(v,x)+t(1-t)\bigl(r-R_N(u,v)\bigr).
\end{equation}

\begin{remark}\label{rem:point-insertion-endpoints}
The identity extends to $t=0$ and $t=1$ by deleting the resulting zero-length segment: it then reduces respectively to
$Q_N(\ee_u-\eta)$ and $Q_N(\ee_v-\eta)$.  In the applications below the inserted edge-nodes satisfy $0<t<1$.
\end{remark}

\begin{proposition}[Transfer under edge compression]\label{prop:general-transfer}
Let $B$ be a weighted network with edge resistances $r_g$, and let $K$ have
the same underlying graph with resistances $0<\tau_g\le\beta r_g$, where
$0<\beta\le1$. Subdivide each edge of $B$ at a finite grid containing both
endpoints, and let $D$ be the resistance diameter of the resulting grid
vertices. Select interior points $z_h,z_k$ on distinct edges of $K$, with
relative coordinates $t_h,t_k$.

For $i=h,k$, choose consecutive grid coordinates bracketing $t_i$, and
let $X_i$ take these two values with mean $t_i$. Put
\[
 \varepsilon_i=\operatorname{Var}(X_i),\qquad
 \rho_i=R_B(u_i,v_i),\qquad
 \nu_i(t)=(1-t)\ee_{u_i}+t\ee_{v_i}.
\]
Choose $X_h,X_k$ independently. In the network obtained by inserting the
two selected points into $K$,
\begin{equation}\label{eq:general-transfer}
 R(z_h,z_k)\le\beta D+C_h+C_k,
\end{equation}
where
\begin{equation}\label{eq:general-error}
 C_i=(\tau_i-\beta r_i)t_i(1-t_i)
       +\beta(r_i-\rho_i)\varepsilon_i.
\end{equation}
Thus $C_h+C_k\le(1-\beta)D$ is sufficient for $R(z_h,z_k)\le D$.
\end{proposition}
\begin{proof}
Thomson's principle gives, for every zero-sum demand,
\begin{equation}\label{eq:PSD}
 Q_K(\zeta)\le\beta Q_B(\zeta).
\end{equation}
By point insertion,
$R(z_h,z_k)=Q_K(\nu_h(t_h)-\nu_k(t_k))+
\sum_{i=h,k}\tau_it_i(1-t_i)$.
Independence and quadratic expansion give
\begin{align}\label{eq:rounding-identity}
 &\mathbb E\!\left[Q_B(\nu_h(X_h)-\nu_k(X_k))
                 +\sum_{i=h,k}r_iX_i(1-X_i)\right]\notag\\
 &\qquad=Q_B(\nu_h(t_h)-\nu_k(t_k))
       +\sum_{i=h,k}\left[r_it_i(1-t_i)-(r_i-\rho_i)\varepsilon_i\right].
\end{align}
Each realization on the left is the resistance between two original grid
vertices, including resistance zero when they coincide. Hence that
expectation is at most $D$. Combining this with \eqref{eq:PSD} gives
\eqref{eq:general-transfer}.
\end{proof}

The two terms in \eqref{eq:general-error} describe different effects:
shortening the edge reduces the first, while rounding introduces the second.
The proposition imposes no degree or regularity condition on the ambient
network. The work specific to line graphs is to bound the sum of these
errors for their degree-dependent compression.

\section{The line graph and its branch core}\label{sec:core}
\subsection{The incidence-star representation and leaf reduction}
Define $\cS(G)$ on $V(G)\sqcup E(G)$ by joining each hub $v$ to its incident
edge-nodes with resistors of resistance $1/d_G(v)$.

\begin{lemma}[Star representation]\label{lem:star}
For $e,f\in E(G)$, $R_{L(G)}(e,f)=R_{\cS(G)}(e,f)$.
\end{lemma}
\begin{proof}
At a hub of degree $d$, Schur elimination gives the boundary Laplacian
\[
 dI_d-(-d\mathbf1)(d^2)^{-1}(-d\mathbf1^{\mathsf T})=dI_d-J_d,
\]
which is the Laplacian of the unit clique $K_d$. This is the classical
star--mesh equivalence \cite{Kennelly1899}. Eliminating every hub produces
$L(G)$: simplicity ensures that two distinct edges share at most one
endpoint. Boundary resistances are preserved.
\end{proof}

Eliminating the edge-nodes instead gives a network on the hubs with
resistances
\begin{equation}\label{eq:hub-weight}
 \omega_{xy}=\frac1{d_G(x)}+\frac1{d_G(y)}.
\end{equation}
Between nonleaf hubs, dangling leaf branches carry no current. Every
remaining edge has $\omega_{xy}\le1$, so Rayleigh monotonicity gives
\begin{equation}\label{eq:hub-rayleigh}
 R_{\cS(G)}(x,y)\le R_G(x,y).
\end{equation}

\begin{proposition}\label{prop:leaf}
Suppose the diameter inequality in Theorem~\ref{thm:main} is known for every
connected graph with fewer edges than $G$.  If $G$ has a leaf, then
\[
   \Dr(L(G))\le \Dr(G).
\]
\end{proposition}

\begin{proof}
Let $e=uv$ with $d_G(v)=1$ and $p=d_G(u)$.  The case $G=K_2$ is immediate,
so assume $p\ge2$.  After the dangling hub $v$ is discarded from $\cS(G)$, by Lemma~\ref{lem:star}
the edge-node $e$ is a leaf joined to $u$ by resistance $1/p$.

Let $f\ne e$, and choose an endpoint $x$ of $f$ with $q=d_G(x)\ge2$.
The resistance triangle inequality, the direct spoke from $x$ to $f$, and
\eqref{eq:hub-rayleigh} give
\[
\begin{aligned}
R_{L(G)}(e,f)
 &=\frac1p+R_{\cS(G)}(u,f)\\
 &\le \frac1p+R_{\cS(G)}(u,x)+\frac1q\\
 &\le 1+R_G(u,x)=R_G(v,x)\le \Dr(G).
\end{aligned}
\]
For $f,g\ne e$, deleting the node $e$ from the line graph and using
Rayleigh monotonicity yields
\[
 R_{L(G)}(f,g)\le R_{L(G-v)}(f,g)
 \le\Dr(G-v)\le\Dr(G).
\]
The last inequality holds because deleting a leaf does not change
resistances between the remaining vertices.
\end{proof}

It remains, by induction, to prove the diameter inequality when the minimum degree is at least two.

\subsection{Degree-two threads and compression}

After the leaf reduction, assume that $\delta(G)\ge2$. If every vertex has degree two, then
$G$ is a cycle and the theorem is immediate.  Otherwise let
\[
   H=\{v\in V(G):d_G(v)\ge3\}.
\]
Every edge lies on a maximal path whose internal vertices have degree two
and whose ends lie in $H$, with coincident ends allowed. We call such a path
a \emph{thread}. An \emph{ordinary thread} has distinct ends; a
\emph{loop thread} returns to the same branch vertex. The threads partition
the edge set.

Delete the interiors and edges of all loop threads, retaining their roots.
The remaining unit-edge graph $\widetilde B$ is the \emph{expanded ordinary
core}. It is connected: a loop thread cannot connect two branch vertices.
If no ordinary thread exists, then $\widetilde B$ is a single branch vertex
and $G$ is a bouquet of loop threads, treated in Section~\ref{sec:loops}.
Each deleted loop is a one-port network, so
\begin{equation}\label{eq:ordinary-preservation}
 R_{\widetilde B}(x,y)=R_G(x,y)
 \quad(x,y\in V(\widetilde B)),
 \qquad D:=\Dr(\widetilde B)\le\Dr(G).
\end{equation}
Loop threads are retained separately as resistor circles when a terminal
lies on them; replacing them by electrical self-loops would lose this
information.

Let an ordinary thread $h$ be
\[
u_h=v_0,v_1,\ldots,v_{L_h}=v_h,
\]
where its internal vertices have degree two.  Put
\begin{equation}\label{eq:thread-pars}
\begin{gathered}
 p_h=d_G(u_h),\qquad q_h=d_G(v_h),\qquad
 a_h=\frac1{p_h},\qquad b_h=\frac1{q_h},\\
 \delta_h=1-a_h-b_h,\qquad
 \tau_h=L_h-\delta_h=L_h-1+a_h+b_h.
\end{gathered}
\end{equation}
Suppressing the degree-two hubs in $\cS(G)$ turns $h$ into a resistor of resistance $\tau_h$.  The edge-node corresponding to $v_mv_{m+1}$ lies at distance $m+a_h$ from $u_h$.  Thus define
\begin{equation}\label{eq:t-nu-k}
 t_h=\frac{m+a_h}{\tau_h},\qquad
 \nu_h(t)=(1-t)\ee_{u_h}+t\ee_{v_h},\qquad
 k_h=\tau_ht_h(1-t_h).
\end{equation}

Suppress the internal vertices of the ordinary threads to obtain a network
$B$ on $H$, with resistance $L_h$ on thread $h$. Let $K$ be the corresponding
compressed star network, with resistance $\tau_h$ on that thread. Parallel
edges are allowed in $B$ and $K$. All endpoint degree labels remain the
degrees in $G$, including incidences belonging to the deleted loop threads.

\begin{lemma}[Short-circuit bounds]\label{lem:rho-bounds}
Let $h$ be an ordinary thread with endpoints $u,v$, length $L$, and full-graph endpoint degrees $p=d_G(u)$, $q=d_G(v)$.  Write $\rho=R_B(u,v)$.  The following short-circuit bounds will be used repeatedly.
\begin{enumerate}[label=\textup{(\roman*)}]
\item If $L=1$, then
\begin{equation}\label{eq:adj-bound}
 \rho\ge \frac{p+q-2}{pq-1}.
\end{equation}
\item If $L\ge2$, then
\begin{equation}\label{eq:thread-rho-bar}
 \rho\ge\bar\rho:=
 \left[\frac1L+1+\frac{(p-2)(q-2)}{p+q-4}\right]^{-1}.
\end{equation}
\end{enumerate}
\end{lemma}

\begin{proof}
For \eqref{eq:adj-bound}, short every vertex other than $u,v$ to one node.
The resulting conductance is at most
\[
 1+\frac{(p-1)(q-1)}{p+q-2}=\frac{pq-1}{p+q-2}.
\]
For \eqref{eq:thread-rho-bar}, let $e\in\{0,1\}$ record whether a direct endpoint edge exists outside the length-$L$ thread.  Wire all vertices outside the entire length-$L$ path to one node, retaining
its internal vertices. The conductance between its ends is at most
\[
 \Gamma_e=\frac1L+e+
 \frac{(p-1-e)(q-1-e)}{p+q-2-2e}.
\]
If there are no vertices outside the path, the conductance is
$1/L+e\le\Gamma_e$, so the same bound holds.
Moreover
\[
 \Gamma_1-\Gamma_0
 =\frac{2pq-3p-3q+4}{(p+q-4)(p+q-2)}>0
 \qquad(p,q\ge3).
\]
Thus $\rho\ge\Gamma_e^{-1}\ge\Gamma_1^{-1}$, which is
\eqref{eq:thread-rho-bar}.  Rayleigh monotonicity proves the claims.
\end{proof}

For the remainder of this section and Section~4, assume that the ordinary
core contains at least one thread.

\subsection{The same-thread case}

\begin{proposition}\label{prop:same-thread}
If two distinct vertices of $L(G)$ lie on the same ordinary thread, their
resistance in $L(G)$ is strictly less than $\Dr(G)$.
\end{proposition}

\begin{proof}
Write the selected edges as $v_mv_{m+1}$ and $v_nv_{n+1}$, with $m<n$, and
put $d=n-m$. Delete the selected thread from $K$ and let $\sigma_K$ be the
resistance between its ends in the remaining network. Set $\sigma_K=\infty$
if the ends are disconnected, and define $\sigma_B$ in the same way.
The parallel rule gives
\[
 R_{L(G)}(e,f)=d-\frac{d^2}{\tau_h+\sigma_K},
\]
with $1/\infty=0$. Since $\tau_h<L_h$ and $\sigma_K\le\sigma_B$, if
$\sigma_B<\infty$ then
\[
 R_{L(G)}(e,f)
 <d-\frac{d^2}{L_h+\sigma_B}
 =R_G(v_m,v_n)\le\Dr(G).
\]
If $\sigma_B=\infty$, every edge of the original thread is a bridge and
$\sigma_K=\infty$. Hence
\[
 R_{L(G)}(e,f)=d\le L_h-1<L_h=R_G(u_h,v_h)\le\Dr(G).
\]
\end{proof}

\subsection{The error terms for the branch core}
Put
\begin{equation}\label{eq:beta-s}
 \beta=\max\left\{\frac23,\max_g\frac{\tau_g}{L_g}\right\},\qquad
 s=1-\beta=\min\left\{\frac13,\min_g\frac{\delta_g}{L_g}\right\},
\end{equation}
where $g$ ranges over all ordinary threads. For distinct selected threads,
deleting the one-port loop interiors and applying point insertion gives
\begin{equation}\label{eq:two-point-exact}
 R_{L(G)}(e,f)=Q_K(\nu_h(t_h)-\nu_k(t_k))+k_h+k_k.
\end{equation}
Apply Proposition~\ref{prop:general-transfer} with $r_i=L_i$ and the unit
path grids. For the edge at index $m_i$, both adjacent grid probabilities
are positive because
\[
 t_i-\frac{m_i}{L_i}=\frac{L_ia_i+m_i\delta_i}{L_i\tau_i}>0,
 \qquad
 \frac{m_i+1}{L_i}-t_i
 =\frac{L_ib_i+(L_i-m_i-1)\delta_i}{L_i\tau_i}>0.
\]
Consequently
\begin{equation}\label{eq:epsilon}
 \varepsilon_i=\left(t_i-\frac{m_i}{L_i}\right)
                 \left(\frac{m_i+1}{L_i}-t_i\right),
\end{equation}
and, with $D=\Dr(\widetilde B)\le\Dr(G)$,
\begin{equation}\label{eq:pre-CB}
 R_{L(G)}(e,f)\le\beta D+C_h+C_k,
\end{equation}
where
\begin{equation}\label{eq:C}
 C_i=(sL_i-\delta_i)t_i(1-t_i)
       +(1-s)(L_i-\rho_i)\varepsilon_i,
 \qquad \rho_i=R_B(u_i,v_i).
\end{equation}
The next section supplies the budget needed to close this comparison.

\section{The full-core budget}

This section proves the estimate that closes \eqref{eq:pre-CB}.  Degree
labels in \eqref{eq:thread-pars} are the degrees in the full graph.  Whenever
the complement of a local configuration is wired to a node $o$, represent
each incidence belonging to a deleted loop thread by an additional unit arm
from its branch endpoint to $o$.  These auxiliary arms restore the full-degree
slots used in the local estimates.  They add conductance and therefore can
only decrease the wired comparison resistance; consequently, a lower bound
proved for the augmented wired network is also a lower bound for the original
unwired resistance used to estimate $D$.

\subsection{A local affine envelope}

For a fixed thread the parameter satisfies
$0\le s\le\min\{1/3,\delta/L\}$, and the following expression extends
affinely to all of $[0,\delta/L]$:
\begin{equation}\label{eq:C-affine}
 C(s)=A+sM_0,\qquad
 A=(L-\rho)\varepsilon-\delta t(1-t),\qquad
 M_0=Lt(1-t)-(L-\rho)\varepsilon.
\end{equation}
Define
\begin{equation}\label{eq:U-H}
 U_s=\frac1{24}+\frac{s}{8},\qquad
 H_s=\frac1{12s}+\frac14\quad(s>0).
\end{equation}
Notice that
\begin{equation}\label{eq:twoU-sH}
   2U_s=sH_s.
\end{equation}

\begin{lemma}[Local envelope]\label{lem:local-envelope}
For every selected ordinary thread,
\begin{equation}\label{eq:LE}
 C_i\le
 \begin{cases}
 U_s,&L_i=1,\\[2mm]
 \displaystyle\frac1{24}+\frac{s}{5}
 =U_s+\frac{3s}{40},&L_i\ge2.
 \end{cases}
\end{equation}
If $s=1/3$, the additional bounds are
\begin{equation}\label{eq:L1-boundary}
 L_i=1:\quad C_i\le\frac{\rho_i}{6}\le\frac D6,
\end{equation}
\begin{equation}\label{eq:L2-boundary}
 L_i=2:\quad C_i\le\frac1{12}<\frac D6.
\end{equation}
For a length-one thread, equality in $C_i\le U_s$ at $0<s<1/3$, or in
$C_i\le\rho_i/6$ at $s=1/3$, holds exactly when
\begin{equation}\label{eq:equality-edge}
 p_i=q_i=3,\qquad\rho_i=\frac12.
\end{equation}
\end{lemma}

The proof is given in Appendix~\ref{app:local-envelope}.

\subsection{A witness thread forces diameter}

\begin{lemma}[Witness diameter]\label{lem:witness-diameter}
Assume $s<1/3$, and choose an ordinary thread $g$ of length $M$ satisfying
$s=\delta_g/M$.  Then
\begin{equation}\label{eq:WD}
   D\ge H_s=\frac1{12s}+\frac14.
\end{equation}
\end{lemma}

\begin{proof}
Write $p,q$ for the endpoint degree labels, and let $e\in\{0,1\}$ indicate
whether a direct endpoint edge exists outside $g$.  Since
$\delta_g\ge1/3$ and $s=\delta_g/M<1/3$, necessarily $M\ge2$.
Let $P_g$ be the expanded
path and put $S=V(\widetilde B)\setminus V(P_g)$.  When
$S\ne\varnothing$, identify all vertices of $S$ to one node $o$ and, if
needed, add one unit arm for each incidence belonging to a deleted loop
thread. Existing arms produced by the contraction are retained, not added
again. Thus each exterior incidence is counted once, giving $p-1-e$ and
$q-1-e$ arms at the two ends.  Put
\begin{equation}\label{eq:witness-vars}
 P=p-1-e,\quad Q=q-1-e,\quad
 \Delta=PQ+e(P+Q),\quad
 T=M+\frac{P+Q}{\Delta},\quad c=\frac e\Delta.
\end{equation}
For the path vertex at distance $i$ from the first endpoint, elementary
series--parallel reduction gives
\begin{equation}\label{eq:wired-Ri}
 R_i=c+\frac{X_i(T-X_i)}{T},\qquad
 X_i=i+\frac Q\Delta.
\end{equation}
For a path vertex $z$ and any $x\in S$, identifying $S$ to $o$ gives
$R(z,o)\le R_{\widetilde B}(z,x)\le D$.  Adding the auxiliary arms can only
decrease the first resistance further.  Hence $D\ge R_i$ whenever
$S\ne\varnothing$.

If $S=\varnothing$ and $e=0$, the ordinary core is the path of length
$M$, so $D=M\ge M/(12\delta_g)+1/4=H_s$.  If $S=\varnothing$ and $e=1$, the
ordinary core is the cycle $C_{M+1}$.  All additional incidences at its
ends come from deleted loop threads, so $p,q\ge4$ and $\delta_g\ge1/2$.
Consequently
\[
 D=\frac{\floor{(M+1)^2/4}}{M+1}
 \ge\frac{M}{12\delta_g}+\frac14;
\]
for $M=2$ this is $2/3\ge7/12$, and for $M\ge3$ the inequality follows
directly from the two parities of $M$.  Thus assume below that
$S\ne\varnothing$.

If $e=0$, choose an integer $i$ nearest the midpoint of the path.  The polynomial identities in
Appendix~\ref{app:WD} prove $R_i\ge H_s$, except when $p=q=3$ and $M$ is
odd.  In that exception, shorting off the path gives
$\rho_g\ge M/(M+1)$.  Taking the path vertex at distance
$d=(M+1)/2$ from an endpoint and using point insertion in the original
network gives
\[
 R_{\widetilde B}(u_g,v_d)
 =\frac{d(M-d)}M+\frac{d^2}{M^2}\rho_g
 \ge\frac{M+1}{4}=H_s.
\]
If $e=1$, choose the integer nearest the continuous maximum in
\eqref{eq:wired-Ri}.  Appendix~\ref{app:WD} again proves $R_i\ge H_s$.
\end{proof}

For later use, for a thread of length $L\ge2$, define
\begin{equation}\label{eq:aL}
   A_L=\frac{\floor{L^2/4}}{L}.
\end{equation}

\begin{lemma}\label{lem:cross-height}
Let $i,j$ be distinct ordinary threads of length at least two, and let $g$
be a witness thread of length $M$.
Then
\[
  D \;\ge\; A_{L_i}+A_{L_j},
  \qquad
  A_M \;\ge\; X-\frac{1}{4M}, \quad X=\frac{1}{12s}.
\]
\end{lemma}

\begin{proof}
Let $\widetilde B'$ be obtained from $\widetilde B$ by identifying the two endpoints of thread $i$,
and likewise of thread $j$. Write $O_h$ for the resulting node, and choose a middle grid vertex
$w_h$ on each thread, so that $R_{C_h}(w_h,O_h)=\lfloor L_h^2/4\rfloor/L_h=A_{L_h}$, where $C_h$
is the cycle into which $h$ is folded. Since $C_h$ meets the rest of $\widetilde B'$ only at
$O_h$, the one-port rule gives
\[
  R_{\widetilde B'}(w_i,w_j)=A_{L_i}+R_{\mathrm{core}}(O_i,O_j)+A_{L_j}\ge A_{L_i}+A_{L_j},
\]
with $R_{\mathrm{core}}(O_i,O_j)=0$ if $O_i=O_j$. As identifying vertices cannot increase
resistances, $D\ge R_{\widetilde B}(w_i,w_j)\ge R_{\widetilde B'}(w_i,w_j)$, proving the first
inequality.

For the second, the endpoints of $g$ have degrees at least three, so $\delta_g\ge 1/3$ and
$X=M/(12\delta_g)\le M/4$. Since $A_M=M/4$ for even $M$ and $A_M=M/4-1/(4M)$ for odd $M$, the
claim follows; equality occurs only when $M$ is odd and $\delta_g=1/3$, i.e.\ both endpoints are
cubic.
\end{proof}

\subsection{The enhanced witness estimate}

The only case not closed by the preceding coarse bounds is when the witness
thread itself is selected and the other selected thread has length one.

\begin{lemma}[Enhanced witness estimate]\label{lem:enhanced}
Assume $s<1/3$, and let the selected witness thread $g$ have length $M$.
Except when
\begin{equation}\label{eq:cubic-M2}
 M=2,\qquad p=q=3,\qquad e=0,
\end{equation}
there is a resistance lower bound $R_*\le D$ such that
\begin{equation}\label{eq:ENH}
 C_g\le U_s+s(R_*-H_s).
\end{equation}
When $P_g$ denotes the expanded witness path,
$S=V(\widetilde B)\setminus V(P_g)\ne\varnothing$, and
$(M,p,q,e)\ne(3,3,3,0)$, $R_*$ is the wired resistance associated with a
specified vertex $z_*\in V(P_g)$ after the auxiliary arms used in the
comparison have been added.
\end{lemma}

Its proof, including the choice of $R_*$ and the small exceptional
configurations, appears in Appendix~\ref{app:ENH}.

\begin{lemma}[Cubic exceptional budget]\label{lem:cubic-exception}
If the selected witness thread satisfies \eqref{eq:cubic-M2} and the other
selected thread has length one, then $C_g+C_k<sD$.
\end{lemma}
The proof is given in Appendix~\ref{app:exceptional}. It compares the
coupled endpoint resistances, which is essential when the two threads meet.

\subsection{The budget lemma}

\begin{lemma}[Full-core budget]\label{lem:CB}
For two distinct selected ordinary threads $h,k$,
\begin{equation}\label{eq:CB}
   C_h+C_k\le sD=(1-\beta)D.
\end{equation}
\end{lemma}

\begin{proof}
First suppose $s<1/3$.  Choose a witness thread $g$ and put
$X=1/(12s)$.  By Lemma~\ref{lem:witness-diameter}, $D\ge H_s$.

If $L_h=L_k=1$, Lemma~\ref{lem:local-envelope} and
\eqref{eq:twoU-sH} immediately give \eqref{eq:CB}.

Suppose exactly one selected thread, say $h$, has length at least two and
$h\ne g$.  Since
$A_{L_h}\ge1/2$ and $A_M\ge X-1/8$,
\[
 D\ge X+\frac38,
\]
whereas the local envelope requires only
\[
 C_h+C_k\le2U_s+\frac{3s}{40}
 =s\left(X+\frac{13}{40}\right).
\]
Since $3/8>13/40$, \eqref{eq:CB} follows in this case.

If both selected threads have length at least two, the local envelope gives
\begin{equation}\label{eq:genuine-thread-budget}
 C_h+C_k\le2U_s+\frac{3s}{20}=s\left(X+\frac25\right).
\end{equation}
Choose a selected thread $h'\ne g$. Lemma~\ref{lem:cross-height} gives
$D\ge A_M+A_{L_{h'}}\ge A_M+1/2$. If $M=2$, then $D\ge1$ and
$X\le1/2$. If $M\ge3$, then $D\ge X-1/(4M)+1/2\ge X+5/12$.
In both cases $D>X+2/5$, so the budget is strict.

It remains that $g$ is the unique selected thread of length at least two
and the other selected thread has length one.  Outside
\eqref{eq:cubic-M2}, Lemmas~\ref{lem:local-envelope} and~\ref{lem:enhanced} yield
\[
 C_g+C_k\le2U_s+s(R_*-H_s)=sR_*\le sD.
\]

The exceptional case \eqref{eq:cubic-M2} is covered by
Lemma~\ref{lem:cubic-exception}. Finally, if $s=1/3$, every ordinary thread
has length at most two, and the bounds in Lemma~\ref{lem:local-envelope}
give $C_h+C_k\le D/3=sD$.
\end{proof}

\begin{corollary}\label{cor:distinct-ordinary}
If two selected vertices of $L(G)$ lie on distinct ordinary threads, then their
resistance in $L(G)$ is at most $\Dr(G)$.
\end{corollary}

\begin{proof}
Equations \eqref{eq:pre-CB} and \eqref{eq:CB} give
$R_{L(G)}(e,f)\le D\le\Dr(G)$.
\end{proof}

\section{Rooted interpolation and loop threads}\label{sec:loops}

We use the following interpolation lemma to estimate an edge-node from a
fixed root. It is stated for a general probability demand; its proof is
given in Appendix~\ref{app:rank-one}.

\begin{lemma}[Rank-one chain interpolation]\label{lem:rank-one}
Let a distinguished edge $uv$ of resistance $L\in\mathbb Z_{\ge2}$ in a finite connected
network $N$, with $u\ne v$, be shortened to $\tau=L-1+a+b$, where
$0<a,b\le1/3$.  Put
$t=(m+a)/\tau$, where $m\in\{0,\ldots,L-1\}$, and
$\nu(x)=(1-x)\ee_u+x\ee_v$.  If $N^-$ denotes the shortened network, then
for every probability vector $\eta$,
\[
Q_{N^-}(\nu(t)-\eta)+\tau t(1-t)
\le\max_{j\in\{m,m+1\}}\left\{
Q_N(\nu(j/L)-\eta)+L\frac jL\left(1-\frac jL\right)
\right\}.
\]
\end{lemma}

The surrounding network and its positive resistances are arbitrary. The
lemma controls extraction at any probability distribution, and the
dominating grid point is always one of the two adjacent points. In
particular, taking $\eta=\ee_r$ gives the rooted estimate used below.

\begin{lemma}[Rooted ordinary-core estimate]\label{lem:rooted-core}
Let $J=\widetilde B$ be the expanded ordinary core, with at least one
ordinary thread, let $K_J$ be its compressed star network, and let $r$ be
a branch vertex. Degree labels are the full degrees in $G$.
For a selected edge-node $z$, write $K_J[z]$ for the network obtained by
inserting $z$ at its position on the corresponding compressed thread.
Put
\[
 E_J(r)=\max_{x\in V(J)}R_J(r,x).
\]
For every edge-node $z$ on an ordinary thread of $J$,
\begin{equation}\label{eq:rooted-core}
   R_{K_J[z]}(r,z)\le E_J(r).
\end{equation}
\end{lemma}

\begin{proof}
Use $\beta,s,C_h$ and the witness thread from Sections~3--4.
Suppose first that $z$ lies on a thread $h$ with $L_h\ge2$.  Starting from
$J$, suppress the internal vertices of $h$, shorten it from $L_h$ to
$\tau_h$, and insert $z$. Suppressing and shortening the other threads then
gives $K_J[z]$. Rayleigh monotonicity
and Lemma~\ref{lem:rank-one}, with $\eta=\ee_r$, give
\[
 R_{K_J[z]}(r,z)\le\max_{j=0,\ldots,L_h}R_J(r,v_j)\le E_J(r).
\]

Now let $L_h=1$.  Write $p=p_h$, $q=q_h$.  With
$\alpha=pq/(p+q)^2$ and
$\rho_h=R_J(u_h,v_h)$, endpoint averaging and \eqref{eq:PSD} give
\[
 R_{K_J[z]}(r,z)\le\beta E_J(r)+C_h,\qquad
 C_h=\frac1{p+q}-\beta\alpha\rho_h.
\]
If $\beta=2/3$, \eqref{eq:L1-boundary} gives
$C_h\le\rho_h/6\le E_J(r)/3$.  If $\beta>2/3$, the local envelope and the
witness-diameter estimate of Lemma~\ref{lem:witness-diameter} give $C_h\le U_s$.  Moreover, the triangle
inequality for the resistance metric gives, for
all $x,y\in V(J)$,
\[
 R_J(x,y)\le R_J(x,r)+R_J(r,y)\le2E_J(r),
\]
so $\Dr(J)\le2E_J(r)$.  Consequently
\[
 2E_J(r)\ge\Dr(J)\ge H_s,\qquad
 sE_J(r)\ge\frac{sH_s}{2}=U_s.
\]
Thus $C_h\le sE_J(r)$ in both cases, proving
\eqref{eq:rooted-core}.
\end{proof}

\begin{corollary}[Strict estimate at a pendant root]\label{cor:strict-root}
Under the assumptions of Lemma~\ref{lem:rooted-core}, suppose that a branch
vertex $u$ has degree one in $J$. Then, for every ordinary edge-node $z$,
\[
 R_{K_J[z]}(u,z)<E_J(u).
\]
\end{corollary}

\begin{proof}
Let $h$ be the unique ordinary thread incident with $u$, with other end $v$.
Every edge of $h$ is a bridge. Let $J'$ be the subgraph on the $v$-side,
obtained by removing $u$ and the internal vertices of $h$, and put
\[
 E'=\max_{x\in V(J')}R_{J'}(v,x),\qquad E=E_J(u)=L_h+E'.
\]
We allow $J'$ to consist only of $v$, in which case $E'=0$. Cut-vertex
additivity bounds a resistance with an endpoint on $h$ by $E$, while the
triangle inequality bounds every resistance within $J'$ by $2E'$. Hence
\begin{equation}\label{eq:pendant-diameter-gap}
 \Dr(J)\le\max\{E,2E'\}<2E.
\end{equation}

If $z$ lies at index $m$ on $h$, the bridge prefix gives
\[
 R_{K_J[z]}(u,z)=m+a_h<m+1\le L_h\le E.
\]
Suppose next that $z$ lies on a distinct thread $q$ of length at least two.
As in Lemma~\ref{lem:rooted-core}, shortening only $q$ and inserting $z$
gives a resistance at most $E$ by Lemma~\ref{lem:rank-one}. Now shorten
$h$. It carries unit current from $u$ to $z$, so this step strictly decreases
the resistance. All remaining shortenings can only decrease it further.

Finally, suppose that $z$ lies on a distinct thread $q$ of length one.
The averaging estimate in Lemma~\ref{lem:rooted-core} gives
\[
 R_{K_J[z]}(u,z)\le\beta E+C_q.
\]
If $\beta>2/3$, the local envelope, the witness bound, and
\eqref{eq:pendant-diameter-gap} imply
\[
 C_q\le U_s=\frac{sH_s}{2}\le\frac{s\Dr(J)}2<sE.
\]
If $\beta=2/3$, both ends of $q$ lie in $J'$, so
$\rho_q\le2E'$ and \eqref{eq:L1-boundary} gives
\[
 C_q\le\frac{\rho_q}{6}\le\frac{E'}3<\frac E3=sE.
\]
Thus the inequality is strict in every case.
\end{proof}

Let a loop thread of length $L\ge3$ be rooted at $u$, and put
$p=d_G(u)$ and
\begin{equation}\label{eq:loop-tau}
   \tau=L-1+\frac2p.
\end{equation}
It becomes a resistor circle of circumference $\tau$ in the star network.

\begin{lemma}[Loop-thread bounds]\label{lem:loop}
For an edge-node $z$ on the loop,
\[
 R_{\cS(G)}(z,u)\le A_L=\frac{\floor{L^2/4}}L.
\]
Equality in the radial bound holds exactly when $L=3$, $p=3$, and $z$
corresponds to the edge opposite the root. For two distinct edge-nodes
$z_i,z_j$ on the same loop thread,
\[
 R_{\cS(G)}(z_i,z_j)<\Dr(G).
\]
\end{lemma}

\begin{proof}
Put $a=1/p\le1/3$.  The allowed coordinates of an edge-node from $u$ are
$x=m+a$, and
\[
 R_{\cS(G)}(z,u)=\frac{x(\tau-x)}\tau.
\]
If $L$ is even, this is at most $\tau/4<L/4=A_L$.  If $L$ is odd, the
discrete maximum is $\tau/4$, and
\[
 A_L-\frac\tau4=\frac{L-1-2aL}{4L}\ge0
\]
because $L\ge3$ and $a\le1/3$. Equality requires $L=3$, $a=1/3$,
and $x=\tau/2$, which is the edge opposite the root.

For the second assertion, let
$d\in\{1,\ldots,L-1\}$ be their integer separation.
Their resistance is
\[
 d-\frac{d^2}{\tau}<d-\frac{d^2}{L},
\]
the resistance between two vertices at separation $d$ on the original
$L$-cycle.
\end{proof}

We now cover every pair involving a loop thread.  For two selected points on
distinct loop threads rooted at $u,v$, Lemma~\ref{lem:loop}, Rayleigh
monotonicity between $K$ and $B$, and cut-vertex additivity give
\[
 R_{L(G)}(e,f)
 \le A_{L_e}+R_B(u,v)+A_{L_f}\le\Dr(G);
\]
if $u=v$, omit the middle term.  Choose vertices $x_e$ and $x_f$ on the two
original loop cycles maximizing their resistances to $u$ and $v$.  By
cut-vertex additivity, $R_G(x_e,x_f)=A_{L_e}+R_B(u,v)+A_{L_f}$ (with the
middle term omitted when $u=v$), and hence this quantity is at most
$\Dr(G)$.

If $e$ lies on a loop rooted at $u$ and $f$ lies on an ordinary thread,
use the expanded ordinary core $J=\widetilde B$.
Lemmas~\ref{lem:rooted-core} and~\ref{lem:loop} give
\[
 R_{L(G)}(e,f)
 =R_{\cS(G)}(e,u)+R_{K_J[f]}(u,f)
 \le A_{L_e}+E_J(u).
\]
Choose a vertex $x_e$ on the loop maximizing $R_G(x_e,u)=A_{L_e}$ and a vertex $x_f\in V(J)$ maximizing $R_G(x_f,u)=E_J(u)$.  By cut-vertex additivity, $R_G(x_e,x_f)=A_{L_e}+E_J(u)$, and hence this quantity is at most $\Dr(G)$.

\section{Proof of Theorem~\ref{thm:main}}

\subsection{Proof of the diameter inequality}

\begin{proof}[Proof of the inequality in Theorem~\ref{thm:main}]
Induct on $|E(G)|$.  The case $G=K_2$ is immediate.  If $G$ has a leaf, apply Proposition~\ref{prop:leaf} and the induction hypothesis.  We may therefore assume $\delta(G)\ge2$.

If every vertex has degree two, $G$ is a cycle and $L(G)\cong G$.  Otherwise decompose $G$ into maximal ordinary and loop threads.  Equal selected vertices have resistance zero.  For two distinct vertices of $L(G)$, there are four exhaustive possibilities:
\begin{enumerate}
\item[(i)] they lie on the same ordinary thread, handled by
Proposition~\ref{prop:same-thread};
\item[(ii)] they lie on distinct ordinary threads, handled by
Corollary~\ref{cor:distinct-ordinary};
\item[(iii)] they lie on the same loop thread,
handled by Lemma~\ref{lem:loop};
\item[(iv)] they lie on distinct threads, at least one of which is a loop thread, handled by the two loop comparisons at the end of
Section~\ref{sec:loops}.
\end{enumerate}
Every pair has resistance at most $\Dr(G)$, and taking the maximum proves the diameter inequality.
\end{proof}

\subsection{Equality and strictness}\label{sec:equality}

\begin{lemma}[Strictness of the interior budget]\label{lem:budget-strict}
Under the hypotheses and notation of the full-core budget,
if $s<1/3$, then $C_h+C_k<sD$ for every pair of distinct ordinary threads.
\end{lemma}
\begin{proof}
Choose a witness thread $g$ of length $M\ge2$. The proof of
Lemma~\ref{lem:CB} is already strict when both selected threads have length
at least two, or when the only such selected thread differs from $g$.
Lemma~\ref{lem:cubic-exception} covers the cubic length-two exception.

For a length-one edge $xy$ with $C_{xy}=U_s$, Lemma~\ref{lem:local-envelope}
gives $d_G(x)=d_G(y)=3$ and $R_{\widetilde B}(x,y)=1/2$. The degree bound
in the simple graph $\widetilde B$ implies
\[
 \frac12\ge\frac1{d_{\widetilde B}(x)+1}
             +\frac1{d_{\widetilde B}(y)+1}\ge\frac12.
\]
Thus $x,y$ have degree three in $\widetilde B$, have identical closed
neighborhoods, and have no attached loop threads. Such an edge can meet
the expanded witness path $P_g$ only as its direct endpoint edge when
$M=2$: the first internal neighbor of one endpoint must also neighbor the
other, and its degree is two.

\emph{The witness and a length-one edge are selected.}
If $C_k<U_s$, the enhanced estimate gives
$C_g+C_k<sR_*\le sD$. Otherwise write $k=xy$ and use the preceding
observation. If $k$ meets $P_g$, write $g=x-w-y$. The common second
neighbor $z$ of $x,y$ is the only possible attachment of the rest of the
core. The two length-two $w$--$z$ paths are in parallel, and the edge $xy$
carries no current by symmetry. Hence $D\ge R(w,z)=1$, while
\[
 s=\frac16,\qquad C_g+C_k=\frac3{40}+\frac1{16}
                         =\frac{11}{80}<sD.
\]
If $k$ is disjoint from $P_g$, put $S=V(\widetilde B)\setminus V(P_g)$.
Except for $(M,p,q,e)=(3,3,3,0)$, Lemma~\ref{lem:enhanced} supplies a
vertex $z_*$ with $R_{\widetilde B}(z_*,S)\ge R_*$. The half-resistance
pair $x,y\in S$ then gives $D\ge R_*+1/8$ by
\eqref{eq:half-resistance}, whereas $C_g+C_k\le sR_*$. In the excluded
case the central wired resistance is $15/16$, so
$D\ge15/16+1/8>H_s=1$ and $C_g+C_k\le2U_s=sH_s<sD$.

\emph{Both selected edges have length one.}
Only $C_h=C_k=U_s$ needs consideration. At most one of these distinct
edges can be the direct endpoint edge of $P_g$. Thus one, say $xy$, is
disjoint from $P_g$ and has resistance $1/2$. The witness computation gives
a path vertex $z_*$ with wired resistance at least $H_s$, except in the
cubic odd case $p=q=3$, $e=0$, when a middle vertex gives
$H_s-1/(4(M+1))$. Here $M\ge3$. In either case,
\[
 D\ge R_{\widetilde B}(z_*,S)+\frac18>H_s,
 \qquad C_h+C_k=2U_s=sH_s<sD.
\]
\end{proof}

\begin{lemma}[Rigidity for distinct ordinary threads]\label{lem:rigidity}
If $G$ has minimum degree at least two and two edge-nodes on distinct
ordinary threads satisfy $R_{L(G)}(e,f)=\Dr(G)$, then $G\cong K_4$.
\end{lemma}
\begin{proof}
Equality in
\[
 R_{L(G)}(e,f)\le\beta D+C_h+C_k\le D\le\Dr(G)
\]
forces $C_h+C_k=sD$ and $D=\Dr(G)$. Lemma~\ref{lem:budget-strict} gives
$s=1/3$. Every term is at most $D/6$, and length-two terms are strictly
smaller by Lemma~\ref{lem:local-envelope}. Both selected threads therefore
have length one, and equality in their bounds forces
\[
 p_h=q_h=p_k=q_k=3,\qquad \rho_h=\rho_k=D=\frac12.
\]
Their rounding variables are uniform on the two endpoints. Equality in
the averaging step of Proposition~\ref{prop:general-transfer} requires
all four cross-endpoint resistances to equal $D$. The endpoints are
distinct, since a shared endpoint would give resistance zero. All six
mutual resistances of the four endpoints thus equal $1/2$.

Any nonadjacent pair among these vertices would have resistance at least
$1/3+1/3=2/3$, by Lemma~\ref{lemma:degree-bound}. They therefore induce
$K_4$. Each already uses all three of its full-graph incidences, so
connectedness forces $G=K_4$.
\end{proof}

\begin{lemma}[Strictness for graphs with a leaf]\label{lem:leaf-strict}
Suppose the equality characterization in Theorem~\ref{thm:main} is known
for all connected graphs with fewer edges than $G$. If $G$ has a leaf, then
$\Dr(L(G))<\Dr(G)$.
\end{lemma}

\begin{proof}
The case $G=K_2$ is immediate. Otherwise, let $e=uv$ be a leaf edge, with
$d_G(v)=1$, put $p=d_G(u)\ge2$, and set $H=G-v$. Resistances on $V(H)$ are
unchanged by deleting $v$.

First consider a pair $e,f$ with $f\ne e$. If $f$ is not a bridge, choose
an endpoint $x$ of $f$ and put $q=d_G(x)\ge2$. In $\cS(G)$ there is an
$x$--$f$ route avoiding the direct spoke, so the parallel rule gives
$R_{\cS(G)}(x,f)<1/q$. The chain in Proposition~\ref{prop:leaf} is therefore
strict. If $f=xy$ is a bridge, choose $x$ on the $u$-side of $G-f$.
Again $q=d_G(x)\ge2$, and
\[
 R_{L(G)}(e,f)\le R_G(u,x)+\frac1p+\frac1q.
\]
This is less than $R_G(v,x)$ unless $p=q=2$. In that last case,
\[
 R_{L(G)}(e,f)\le R_G(v,x)<R_G(v,y)=R_G(v,x)+1\le\Dr(G).
\]
Thus every pair involving $e$ has resistance strictly less than $\Dr(G)$.

Suppose, for a contradiction, that $\Dr(L(G))=\Dr(G)$. A maximizing pair
then avoids $e$, so deleting $e$ from $L(G)$ and using the already proved
diameter inequality gives
\[
 \Dr(G)\le\Dr(L(H))\le\Dr(H)\le\Dr(G).
\]
All terms are equal. By the induction hypothesis, $H$ is a cycle or $K_4$.
Both are vertex-transitive, so the resistance eccentricity of the attachment
vertex $u$ equals $\Dr(H)$. Adding the leaf gives
\[
 \Dr(G)\ge1+\max_{x\in V(H)}R_H(u,x)=1+\Dr(H)>\Dr(G),
\]
a contradiction.
\end{proof}

\begin{proof}[Completion of the proof of Theorem~\ref{thm:main}]
For a cycle, $L(C_n)\cong C_n$. For $K_4$, the two nonterminal vertices
have equal potential, so the edge between them carries no current. The
direct terminal edge is then in parallel with two length-two paths, giving
resistance $1/2$. Its line graph is
$K_{2,2,2}$. Two vertices in the same part have resistance $1/2$, while
vertices in different parts have resistance $5/12$. To check the latter
value, impose terminal potentials $1$ and $0$. The other vertices in those
two parts have potentials $2/5$ and $3/5$, and the two remaining vertices
have potential $1/2$. The source current is $12/5$.
Thus both cycles and $K_4$ give equality.

For necessity, induct on $|E(G)|$. The graph $K_2$ is strict. Suppose
$\Dr(L(G))=\Dr(G)$, and assume the equality characterization for smaller
graphs. Lemma~\ref{lem:leaf-strict} excludes leaves. If every vertex has
degree two, $G$ is a cycle. Otherwise, choose distinct edge-nodes $e,f$ with
$R_{L(G)}(e,f)=\Dr(G)$.

Proposition~\ref{prop:same-thread} and Lemma~\ref{lem:loop} exclude pairs on
the same ordinary or loop thread. We next exclude the two remaining
configurations involving loops.

\emph{Two distinct loop threads.} Let their roots be $u,v$. Equality in a
radial bound is possible only for a triangular loop at a cubic root, by
Lemma~\ref{lem:loop}. If either radial bound is strict, the desired
strictness follows at once. If $u=v$, two loops use at least four incidences at
their common root, so at least one radial bound is strict. If $u\ne v$ and
both radial bounds are equalities, each cubic root has a unique ordinary
thread. The ordinary thread at $u$ is a bridge and carries unit current in
the $u$--$v$ problem. Shortening it strictly decreases the core resistance;
the other shortenings cannot increase it. In either case,
\[
 R_{L(G)}(e,f)<A_{L_e}+R_B(u,v)+A_{L_f}\le\Dr(G),
\]
where the middle term is omitted when $u=v$.

\emph{A loop thread and an ordinary thread.} If the radial bound on the
loop is strict, the comparison at the end of Section~\ref{sec:loops} is
strict. Otherwise the loop is triangular at a cubic root $u$. It uses two
incidences, so $u$ has degree one in the expanded ordinary core $J$.
Corollary~\ref{cor:strict-root} now gives
\[
 \begin{aligned}
 R_{L(G)}(e,f)
 &=R_{\cS(G)}(e,u)+R_{K_J[f]}(u,f)\\
 &<A_{L_e}+E_J(u)\le\Dr(G).
 \end{aligned}
\]

Only two distinct ordinary threads remain. Lemma~\ref{lem:rigidity} then
gives $G\cong K_4$, completing the induction.
\end{proof}

\section{Concluding remarks}

Theorem~\ref{thm:main} gives
$\Dr(G)\ge\Dr(L(G))\ge\cdots\ge\Dr(L^m(G))$ whenever the intermediate
graphs have at least one edge. The equality classification is exact, but
there is no positive uniform additive or multiplicative gap outside its
two families.

For the additive assertion, let $A$ be the adjacency matrix of $L(K_n)$,
$n\ge4$. The unsigned incidence identities
$M^{\mathsf T}M=A+2I$ and $MM^{\mathsf T}=(n-2)I+J$ give nonzero
Laplacian eigenvalues $n$ and $2(n-1)$. Consequently
\[
 \cL_{L(K_n)}^+=\frac{n+2}{2n(n-1)}I
                    +\frac1{2n(n-1)}A+\gamma J.
\]
Adjacent and nonadjacent pairs have resistances $(n+1)/(n(n-1))$ and
$(n+2)/(n(n-1))$, respectively. Thus, for $n\ge5$,
\begin{equation}\label{eq:complete-line-diameter}
 \Dr(L(K_n))=\frac{n+2}{n(n-1)},\qquad
 \Dr(K_n)-\Dr(L(K_n))=\frac{n-4}{n(n-1)}\longrightarrow0.
\end{equation}

For the multiplicative assertion, attach a path of length $t\ge1$ to a
triangle at one endpoint, obtaining $G_t$. Cut-vertex additivity gives
$\Dr(G_t)=t+2/3$. Its line graph consists of $K_4$ minus one edge, with a
path of length $t-1$ attached at one of the nonadjacent vertices. That
four-vertex core has resistance diameter one, attained by the nonadjacent
pair, so $\Dr(L(G_t))=t$. Hence
$\Dr(L(G_t))/\Dr(G_t)=t/(t+2/3)\to1$.

The comparison proposition and the interpolation lemma apply to weighted
networks without assumptions on the surrounding topology. The full-core
budget proved here uses the unit-edge model, simplicity, and full endpoint
degrees. A weighted analogue therefore calls for a corresponding budget;
the transfer formula specifies exactly the error terms that such a result
would need to control.

\section*{Data availability}
The article is theoretical. The polynomial identities and reduced-network
calculations used in the proofs are included in the appendices.

\section*{Acknowledgement}
The research of Xiang-Feng Pan is supported by the Excellent University Research and Innovation Team in Anhui Province (No. 2024AH010002) and the University Natural Science Research Project of Anhui Province (No. 2025AHGXZK30908).
The research of Zhen-Mu Hong is supported by Natural Science Foundation of China (No. 12371338) and Outstanding Youth Scientific Research Projects of Anhui Provincial Department of Education (No. 2022AH030073).

\appendix

\section{Proof of the local envelope}\label{app:local-envelope}

\begin{proof}
First let $L=1$.  Put $S=p+q$, $P=pq$, and
$\alpha=P/S^2$.  Here
\[
 C(s)=\frac1S-(1-s)\alpha\rho.
\]
It is enough to check $s=0$ and $s=1/3$.  Substitute
\eqref{eq:adj-bound} and multiply by the positive denominators.  The two
required numerators are
\[
 \Phi_0=P(S^2-48)-S^2+24S,\qquad
 \Phi_1=P(S^2-4S-16)-S^2+12S.
\]
Writing $p=3+x$, $q=3+y$, and $u=x+y$, they become
\begin{align*}
\Phi_0={}&84u+44(x^2+y^2)+3(x^3+y^3)+76xy
          +21xyu+xyu^2,\\
\Phi_1={}&60u+32(x^2+y^2)+3(x^3+y^3)+60xy
          +17xyu+xyu^2.
\end{align*}
Both are nonnegative.

Now assume $L\ge2$.  Since both sides of the desired estimate are affine in
$s$, it suffices to check $s=0$ and $s=\delta/L$.  From
$m=(L-\delta)t-a$ one obtains
\begin{equation}\label{eq:epsilon-factor}
 \varepsilon=\frac{(a+\delta t)(b+\delta(1-t))}{L^2}.
\end{equation}
Moreover,
\[
 A''(t)=2\delta\left(1-\frac{(L-\rho)\delta}{L^2}\right)>0.
\]
Consequently $A$ in \eqref{eq:C-affine} is convex as a function of $t$, so
its maximum over $m=0,\ldots,L-1$ occurs at $m=0$ or $m=L-1$.  At either
end, with the orientation chosen appropriately,
\begin{equation}\label{eq:A-end-identity}
 L\varepsilon-\delta t(1-t)=\frac{ab}{\tau},\qquad
 A=\frac{ab}{\tau}-\rho\varepsilon.
\end{equation}
For $L\ge3$ this is at most
\[
 \frac1{pq(L-1)+p+q}\le\frac1{24}.
\]
For $L=2$, the same argument is immediate unless, after assuming $p\le q$,
$(p,q)$ is $(3,3),(3,4)$, or $(3,5)$.  Substitution of
\eqref{eq:thread-rho-bar} gives respectively
\[
 \max A\le\frac{11}{300},\qquad
 \frac{10}{361},\qquad
 \frac{35}{1587},
\]
all smaller than $1/24$.

At the other endpoint of the affine interval,
\begin{equation}\label{eq:E-end}
 E:=C(\delta/L)=\frac{\tau}{L}(L-\rho)\varepsilon.
\end{equation}
For $L\ge4$, use $\varepsilon\le1/(4L^2)$ and $\delta\ge1/3$.
After multiplication by $120L^2$, the required difference is at least
\[
 5L^2-30L+\delta(24L+30)
 \ge5L^2-22L+10\ge0.
\]
For $L=3$, the same rough estimate works when $\delta\ge15/34$; the only
remaining degree pairs are shown in the first two rows below.  For $L=2$ it
works when $\delta\ge20/39$; the remaining pairs are the last five rows.
The third column bounds the maximum over the grid positions, using
\eqref{eq:thread-rho-bar} in \eqref{eq:E-end}.
\[
\begin{array}{c@{\qquad}c@{\qquad}c@{\qquad}c}
\toprule
L&(p,q)&\text{upper bound for }E&1/24+\delta/(5L)\\
\midrule
3&(3,3)&2/33&23/360\\
3&(3,4)&595/10044&5/72\\
2&(3,3)&3/40&3/40\\
2&(3,4)&55/741&1/12\\
2&(3,5)&91/1242&53/600\\
2&(3,6)&5/69&11/120\\
2&(4,4)&1/15&11/120\\
\bottomrule
\end{array}
\]
This proves both affine endpoint bounds and hence \eqref{eq:LE}.
For $0<s<1/3$, the gap $U_s-C_i$ for $L_i=1$ is affine and
nonnegative at both endpoints $s=0,1/3$. Equality at an interior value
forces both endpoint gaps to vanish. The displayed polynomials then give
$p_i=q_i=3$, and
\[
 U_s-C_i=\frac{1-s}{8}(2\rho_i-1).
\]
Thus equality holds exactly when $\rho_i=1/2$.

Finally consider $s=1/3$, equivalently $\beta=2/3$.  The condition
$\delta_i/L_i\ge1/3$ forces $L_i\le2$.  For $L=1$, the adjacent lower bound
gives
\begin{equation}\label{eq:L1-boundary-proof}
   C_i\le\frac{\rho_i}{6}\le\frac D6.
\end{equation}
Indeed the first inequality reduces to
\[
(p+q-2)(p^2+6pq+q^2)\ge6(p+q)(pq-1).
\]
With $p=3+x,q=3+y,u=x+y$, the difference is
$12u+10(x^2+y^2)+u(x^2+y^2)\ge0$.

For $L=2$, write $a=1/p$, $b=1/q$, and $S=a+b$.  The condition
$\tau/L\le2/3$ is $S\le1/3$.  For $m=0$,
$t=a/(1+S)$ and $\varepsilon=t(1/2-t)$, so
\begin{align}
 C&\le(S-1/3)t(1-t)+\frac43t(1/2-t)\notag\\
  &=\frac{a(b+1/3)}{1+S}
  \le\frac{(S+1/3)^2}{4(1+S)}\le\frac1{12}.
 \label{eq:L2-boundary-proof}
\end{align}
The last step is equivalent to
$(3S-1)(S+2/3)\le0$; $m=1$ is symmetric.  A middle thread vertex satisfies
\[
 R_{\widetilde B}(v_1,u)=\frac12+\frac{\rho_i}{4}\ge\frac12,
\]
and $\rho_i>0$ makes $1/12<D/6$. For $L_i=1$, equality in the
degree polynomial requires $p_i=q_i=3$; equality in the resistance
substitution then requires $\rho_i=1/2$.
\end{proof}

\section{Polynomial verification for the witness diameter}\label{app:WD}

We verify the sign assertions in Lemma~\ref{lem:witness-diameter}.  All
auxiliary variables $x,y,z,w$
introduced below are nonnegative integers.

First let $e=0$.  Put
\[
A=(P+1)(Q+1),\qquad d_0=PQ-1,\qquad V=P+Q+MPQ.
\]
Then $\delta=d_0/A$.  If $N$ is the numerator of the selected wired
resistance $R_i=N/V$, the inequality $R_i\ge H_s$ is equivalent to
\begin{equation}\label{eq:Omega-WD}
 \Omega:=12d_0N-MAV-3d_0V\ge0.
\end{equation}
For $M=2n$, put $P=2+x,Q=2+y,n=1+z$ and
$N=(nP+1)(nQ+1)$.  Direct expansion gives
\begin{align*}
\Omega={}&45(x+y)+102z(x+y)+12(x^2+y^2)+54xy
 +48z^2(x+y)+42z(x^2+y^2)+146xyz\\
&+11xy(x+y)+24z^2(x^2+y^2)+80xyz^2
 +42xyz(x+y)+2x^2y^2\\
&+28xyz^2(x+y)+10x^2y^2z+8x^2y^2z^2\ge0.
\end{align*}
For $M=2n+1$, assume $P\le Q$, use
$N=(nP+1)((n+1)Q+1)$, and again put
$P=2+x,Q=2+y,n=1+z$.  Then
\begin{align*}
\Omega={}&-36+48x+84y+150z(x+y)+15x^2+39y^2+102xy
 +48z^2(x+y)\\
&+66z(x^2+y^2)+226xyz+21x^2y+33xy^2
 +24z^2(x^2+y^2)+80xyz^2\\
&+70xyz(x+y)+6x^2y^2+28xyz^2(x+y)
 +18x^2y^2z+8x^2y^2z^2.
\end{align*}
All nonconstant coefficients are nonnegative.  If $(x,y)\ne(0,0)$, the
smallest allowed contribution is already $84+39>36$; $(x,y)=(0,0)$ is the
symmetric cubic odd case treated separately in the proof.

Now let $e=1$.  Put
\[
\begin{gathered}
 P=1+x,\quad Q=1+y,\quad M=2+z,\\
 \Delta=PQ+P+Q,\quad A=(P+2)(Q+2),\quad
 V=P+Q+M\Delta.
\end{gathered}
\]
Rounding to the nearest integer in \eqref{eq:wired-Ri} gives
\[
 \max_iR_i\ge\frac{V^2+4V-\Delta^2}{4\Delta V}.
\]
Indeed, the continuous maximizer corresponds to
\[
 i_0=\frac M2+\frac{P-Q}{2\Delta}.
\]
Since $|P-Q|\le\Delta$ and $M\ge2$, one has $i_0\in[0,M]$; hence an
admissible integer lies within distance $1/2$ of $i_0$.  Moreover,
\[
 X(T-X)=\frac{T^2}{4}-\left(X-\frac T2\right)^2,
\]
so the nearest admissible integer makes the last square at most $1/4$;
substitution in \eqref{eq:wired-Ri} gives the displayed bound.
The desired inequality is
\begin{equation}\label{eq:Omega1-WD}
 \Omega_1:=3(V^2+4V-\Delta^2)-MAV-3\Delta V\ge0.
\end{equation}
For $z\ge1$, write $z=1+w$.  Expansion gives
\begin{align*}
\Omega_1={}&72+27w+93(x+y)+69w(x+y)+30(x^2+y^2)+102xy
 +9w^2(x+y)\\
&+33w(x^2+y^2)+106xyw+27xy(x+y)+6w^2(x^2+y^2)
 +18xyw^2\\
&+35xyw(x+y)+6x^2y^2+7xyw^2(x+y)
 +9x^2y^2w+2x^2y^2w^2>0.
\end{align*}
For $z=0$, assume $x\le y$.  The condition $s<1/3$ leaves only
$x=0$; or $x=1,y\le8$; or $x=2,y\le4$.  In these three cases
\eqref{eq:Omega1-WD} is respectively
\[
45+33y+3y^2,\qquad 81+46y+y^2,\qquad
123+57y-3y^2,
\]
which is positive on the stated ranges.

\section{Polynomial verification for the enhanced estimate}\label{app:ENH}

\subsection{Reduction to the polynomial estimates}
\begin{proof}
Because $s=\delta_g/M$,
\[
 C_g=(1-s)(M-\rho_g)\varepsilon_g.
\]
Let $P_g$ be the expanded witness path and put
$S=V(\widetilde B)\setminus V(P_g)$.  If $S=\varnothing$, take $R_*=D$.
When $e=0$, $D=M$ and
$D-H_s\ge3M/4-1/4\ge3/40$.  When $e=1$, the ordinary core is
$C_{M+1}$, while $p,q\ge4$ and $\delta_g\ge1/2$.  Hence
\[
D=A_{M+1}\ge H_s+\frac3{40}:
\]
for $M=2$ the difference $D-H_s$ is at least $1/12$, for $M=3$ it is at
least $1/4$,
and for $M\ge4$ use $A_{M+1}\ge M/4\ge M/6+13/40$.
The local envelope now gives
\[
C_g\le U_s+\frac{3s}{40}
\le U_s+s(D-H_s),
\]
so assume below that $S\ne\varnothing$.

Use $\varepsilon_g\le1/(4M^2)$ and the $e$-sensitive shorting bound
\begin{equation*}
 \rho_g\ge\bar\rho_g:=
 \left[\frac1M+e+\frac{PQ}{P+Q}\right]^{-1},
\end{equation*}
where $P,Q$ are as in \eqref{eq:witness-vars}.  It is sufficient that
\begin{equation}\label{eq:Phi-enh}
 \Phi:=24\delta_gR_*M^2-M^3-3\delta_gM^2
       -6(M-\delta_g)(M-\bar\rho_g)\ge0.
\end{equation}
Indeed,
\[
\frac{\Phi}{24M^3}
=sR_*-U_s-
\frac{(M-\delta_g)(M-\bar\rho_g)}{4M^3},
\]
while
$C_g\le(M-\delta_g)(M-\bar\rho_g)/(4M^3)$.

For $e=0$, write $M=2n$ or $2n+1$, take $z_*=v_n$ and $R_*=R_n$,
orienting the thread so that $P\le Q$ in the odd case.  For $e=1$, choose
$i_*$ with $R_{i_*}=\max_{0\le i\le M}R_i$, and take
$z_*=v_{i_*}$ and $R_*=R_{i_*}$.  Because $S$ was wired before the auxiliary
arms were added, these choices satisfy $R_*\le D$.
The calculations below expand \eqref{eq:Phi-enh}; the only failures are
\eqref{eq:cubic-M2} and the initial wired bound for $M=3,p=q=3,e=0$.  In the latter
case choose instead the
endpoint-to-distance-two pair.  Then $R_*=H_s=1$, while
$\rho_g\ge3/4$ and $\varepsilon_g\le1/36$, so
\[
 C_g\le\frac89\cdot\frac94\cdot\frac1{36}
 =\frac1{18}=U_s.
\]
Thus \eqref{eq:ENH} holds there as well.
\end{proof}

\subsection{Exact expansions}

We verify \eqref{eq:Phi-enh}.  For $e=0$, use the notation of
Appendix~\ref{app:WD}, including its parity-dependent definition of $N$;
then $\bar\rho=M(P+Q)/V$.  The variables $x,y,z$ below are again
nonnegative integers.  Multiplying $\Phi$ by the positive factor $AV/M^2$ gives
\[
 \Psi=24d_0N-MAV-3d_0V-6PQ(MA-d_0).
\]
For $M=2n$, $P=2+x,Q=2+y,n=1+z$, expansion gives
\begin{align*}
\Psi_{\rm even}={}&-36+144z^2+93(x+y)+210z(x+y)
 +36(x^2+y^2)+132xy+216z^2(x+y)\\
&+90z(x^2+y^2)+302xyz+35xy(x+y)+72z^2(x^2+y^2)
 +260xyz^2\\
&+90xyz(x+y)+8x^2y^2+76xyz^2(x+y)
 +22x^2y^2z+20x^2y^2z^2.
\end{align*}
This is nonnegative except at $x=y=z=0$, namely the cubic $M=2$ case.
For $M=2n+1$ with the same substitutions,
\begin{align*}
\Psi_{\rm odd}={}&-72+144z+144z^2+132x+204y+426z(x+y)
 +51x^2+99y^2+258xy\\
&+216z^2(x+y)+162z(x^2+y^2)+562xyz+63x^2y+87xy^2
 +72z^2(x^2+y^2)\\
&+260xyz^2+166xyz(x+y)+18x^2y^2+76xyz^2(x+y)
 +42x^2y^2z+20x^2y^2z^2.
\end{align*}
With $x\le y$, the only negative case is $x=y=z=0$, namely the initial wired
choice for $M=3,p=q=3$; it is replaced in
Lemma~\ref{lem:enhanced} by an endpoint-to-thread pair.

For $e=1$, use $P=1+x,Q=1+y,M=2+z$ and the notation of
Appendix~\ref{app:WD}.  After multiplication by $AV/M^2$ and insertion of
the nearest-integer resistance lower bound, it is enough that
\[
\Psi_1=6(V^2+4V-\Delta^2)-MAV-3\Delta V-6\Delta(MA-\Delta)\ge0.
\]
Its expansion is
\begin{align*}
\Psi_1={}&36+45z+27z^2+45(x+y)+99z(x+y)+18(x^2+y^2)+38xy
 +45z^2(x+y)\\
&+45z(x^2+y^2)+142xyz+11xy(x+y)+18z^2(x^2+y^2)
 +60xyz^2\\
&+45xyz(x+y)+2x^2y^2+19xyz^2(x+y)
 +11x^2y^2z+5x^2y^2z^2,
\end{align*}
which is strictly positive.

\section{Proof of the rank-one chain interpolation lemma}\label{app:rank-one}

We prove Lemma~\ref{lem:rank-one} in its stated generality.  Put
\[
 z=\ee_v-\ee_u,\qquad
 \lambda=\frac1\tau-\frac1L=\frac\delta{L\tau},\qquad
 \delta=1-a-b=L-\tau,
 \qquad
 \rho=z^{\mathsf T}\cL_N^+z=R_N(u,v).
\]
Shortening the distinguished edge from resistance $L$ to $\tau$ adds a
parallel conductance $\lambda$, and hence
\begin{equation}\label{eq:rank-one-laplacian}
 \cL_{N^-}=\cL_N+\lambda zz^{\mathsf T}.
\end{equation}
Both Laplacians have kernel $\operatorname{span}\{\mathbf1\}$.  Restricting
\eqref{eq:rank-one-laplacian} to the zero-sum subspace, where $\cL_N$ is
invertible, the Sherman--Morrison formula gives
\begin{equation}\label{eq:rank-one-pseudoinverse}
 \cL_{N^-}^+
 =\cL_N^+-
 \frac{\lambda\,\cL_N^+zz^{\mathsf T}\cL_N^+}{1+\lambda\rho}
 \qquad\text{on }\R_0^{V(N)}.
\end{equation}
Consequently, for every $\zeta\in\R_0^{V(N)}$,
\begin{equation}\label{eq:rank-one-energy-update}
 Q_{N^-}(\zeta)
 =Q_N(\zeta)-
 \frac{\lambda\,\langle\zeta,z\rangle_N^2}{1+\lambda\rho}.
\end{equation}
On the other hand, completing the square gives
\[
 \min_{y\in\R}\bigl\{Q_N(\zeta-yz)+\lambda^{-1}y^2\bigr\}
 =Q_N(\zeta)-
 \frac{\lambda\,\langle\zeta,z\rangle_N^2}{1+\lambda\rho},
\]
the minimizer being
$y=\lambda\langle\zeta,z\rangle_N/(1+\lambda\rho)$.  Taking
$\zeta=\nu(t)-\eta$, observing that
$\zeta-yz=\nu(t-y)-\eta$, and renaming $t-y$ as $x$ yields the rank-one
variational identity
\begin{equation}\label{eq:rank-one-var}
Q_{N^-}(\nu(t)-\eta)
=\min_{x\in\R}\left\{Q_N(\nu(x)-\eta)+\lambda^{-1}(t-x)^2\right\}.
\end{equation}
Let $x_*$ be the minimizer, and set
\[
 F(x)=Q_N(\nu(x)-\eta)+Lx(1-x),\qquad
 r_0=m+a+\frac\delta2.
\]
The following identity is obtained by comparing coefficients:
\begin{equation}\label{eq:rank-square}
\lambda^{-1}(t-x)^2+\tau t(1-t)-Lx(1-x)
=\frac{(Lx-r_0)^2}{\delta}-\frac\delta4.
\end{equation}
Thus, with
$\mathcal T=Q_{N^-}(\nu(t)-\eta)+\tau t(1-t)$,
\begin{equation}\label{eq:T-rank}
 \mathcal T=F(x_*)+\frac{(Lx_*-r_0)^2}{\delta}-\frac\delta4.
\end{equation}

Put $\phi=\cL_N^+z$, and define
\[
 c=(\ee_u-\eta)^{\mathsf T}\phi.
\]
Here $\phi_v-\phi_u=\rho$.  By the maximum principle, every value of $\phi$
lies in $[\phi_u,\phi_v]$; since $\eta$ is a probability vector,
$\eta^{\mathsf T}\phi$ is a convex combination of these values.  Hence
$-\rho\le c=\phi_u-\eta^{\mathsf T}\phi\le0$.  Write $c=-\rho\xi$, $0\le\xi\le1$, and put
\[
 \sigma=\frac{L-\rho}{L}\in[0,1],\qquad
 \Delta_0=L-\delta\sigma,\qquad r=Lx_*.
\]
The stationarity equation in \eqref{eq:rank-one-var} becomes
\begin{equation}\label{eq:r-star}
r=\frac{L\bigl(m+a+\delta(1-\sigma)\xi\bigr)}{\Delta_0}.
\end{equation}
In particular $m\le r\le m+1$, since
\begin{align*}
r-m&=\frac{La+m\delta\sigma+L\delta(1-\sigma)\xi}{\Delta_0}\ge0,\\
m+1-r&=\frac{Lb+\delta\sigma(L-m-1)
+L\delta(1-\sigma)(1-\xi)}{\Delta_0}\ge0.
\end{align*}

Since $F''=-2L\sigma$, the stationarity equation and
\eqref{eq:T-rank} yield, for integers $i$,
\begin{equation}\label{eq:grid-diff}
F(i/L)-\mathcal T
=\frac\delta4-\frac{(i-r_0)^2}{\delta}
+\frac{\Delta_0}{\delta L}(i-r)^2.
\end{equation}
At $i=m,m+1$, respectively,
\begin{align}
F(m/L)-\mathcal T
&=\frac1\delta\left[\frac{\Delta_0}{L}(r-m)^2-a(1-b)\right],
\label{eq:left-grid}\\
F((m+1)/L)-\mathcal T
&=\frac1\delta\left[\frac{\Delta_0}{L}(m+1-r)^2-b(1-a)\right].
\label{eq:right-grid}
\end{align}
Define
\[
A=\sqrt{\frac{La(1-b)}{\Delta_0}},\qquad
B=\sqrt{\frac{Lb(1-a)}{\Delta_0}}.
\]
Since $(r-m)+(m+1-r)=1$, one of
\eqref{eq:left-grid}--\eqref{eq:right-grid} is nonnegative whenever
$A+B\le1$.

For $L\ge3$, put
\[
S=\left(\sqrt{a(1-b)}+\sqrt{b(1-a)}\right)^2,\quad
X=\sqrt{(1-a)(1-b)},\quad Y=\sqrt{ab}.
\]
Then $1-S=(X-Y)^2$ and $X^2-Y^2=\delta$.  Since $a,b\le1/3$, one has
$X\ge2Y$, and consequently
\[
1-S=\frac{\delta^2}{(X+Y)^2}\ge\frac\delta3.
\]
Thus $LS\le L-\delta\le\Delta_0$, which is exactly $(A+B)^2\le1$.

It remains to consider $L=2$; reversing the edge reduces us to $m=0$.
If $A+B\le1$ we are done.  Otherwise
\begin{equation}\label{eq:sigma0}
 \sigma>\sigma_0:=\frac{2(1-S)}\delta.
\end{equation}
Since $\sigma\le1$, this also gives
$0\le\sigma_0<1$.
From \eqref{eq:r-star},
\[
r\le r_+(\sigma):=
\frac{2(a+\delta(1-\sigma))}{2-\delta\sigma}.
\]
At $\sigma=\sigma_0$, where $A+B=1$, the inequality
$r_+(\sigma_0)\le A=1-B$ is equivalent to
\begin{equation}\label{eq:L2-boundary-rank}
a+b-3ab+3\sqrt{ab(1-a)(1-b)}\le1.
\end{equation}
This follows from
\[
(1-a-b+3ab)^2-9ab(1-a)(1-b)
=\delta(\delta-3ab)\ge0,
\]
because $a,b\le1/3$.  Finally,
\begin{align*}
&(2-\delta\sigma)^2
\left[(1-r_+(\sigma))^2-B^2\right]\\
&\qquad=\delta\left[\delta\sigma^2+2b(3-a)\sigma-4b\right].
\end{align*}
The bracket is increasing in $\sigma\ge0$ and is nonnegative at
$\sigma_0$ by \eqref{eq:L2-boundary-rank}.  In addition,
\[
 1-r_+(\sigma)=\frac{2b+\delta\sigma}{2-\delta\sigma}>0.
\]
Hence the preceding squared inequality has the positive-sign branch, and
$1-r\ge1-r_+(\sigma)\ge B$, so \eqref{eq:right-grid} is nonnegative.
This proves that at least one adjacent grid value is at least $\mathcal T$,
and completes the proof of Lemma~\ref{lem:rank-one}.

\clearpage
\section{Proof of the cubic exceptional budget}\label{app:exceptional}

\begin{proof}[Proof of Lemma~\ref{lem:cubic-exception}]
Assume \eqref{eq:cubic-M2}.  Write $g=u-w-v$.  Then
$s=1/6$, $H_s=3/4$, $\rho_g\ge2/3$, and
\begin{equation}\label{eq:Cg-cubic}
   C_g=\frac{2-\rho_g}{20}\le\frac1{15}.
\end{equation}
Let the selected length-one thread $k=xy$ have endpoint degrees $r,t$, put
$S_0=r+t$, $\alpha_k=rt/S_0^2$, and write $\rho_k=R_B(x,y)$.

\textbf{Subcase A: }$\{x,y\}\cap\{u,v\}=\varnothing$.  Put
$\mathcal S=V(\widetilde B)\setminus\{u,w,v\}$ and identify all vertices of
$\mathcal S$ to a single node $s_*$. Wiring only lowers resistances;
after the auxiliary arms for deleted loop
incidences are added, each of $u$ and $v$ is joined to $s_*$ by two parallel
unit arms, while $w$ is joined to $u$ and $v$.  By symmetry the resistance
from $w$ to $s_*$ in this augmented wired network is
\[
\frac{(\tfrac32)(\tfrac32)}{\tfrac32+\tfrac32}=\frac34.
\]
Deleting the auxiliary arms and undoing the wiring cannot decrease resistance,
so Lemma~\ref{lem:set-contraction} gives
\begin{equation*}
R_{\widetilde B}(w,\mathcal S)
=R_{\widetilde B/\mathcal S}(w,s_*)\ge\frac34.
\end{equation*}
Apply \eqref{eq:variance-bound} in Lemma~\ref{lem:set-contraction}
to $N=\widetilde B$, $z=w$, and the probability vector
$\nu=\frac{r}{S_0}\ee_x+\frac{t}{S_0}\ee_y$ supported on
$\{x,y\}\subseteq\mathcal S$.
Its variance term is
$\frac12\sum_{i,j}\nu_i\nu_jR_{\widetilde B}(i,j)=\frac{rt}{S_0^2}\rho_k=\alpha_k\rho_k$
(since $R(x,x)=R(y,y)=0$).  Hence
\begin{equation}\label{eq:subcase-A-diameter}
 D\ge\max\{R(w,x),R(w,y)\}
 \ge R_{\widetilde B}(w,\mathcal S)+\alpha_k\rho_k
 \ge\frac34+\alpha_k\rho_k.
\end{equation}
By \eqref{eq:adj-bound},
$\alpha_k\rho_k\ge(S_0-2)/S_0^2$. Combining
$C_k=1/S_0-\tfrac56\alpha_k\rho_k$, \eqref{eq:Cg-cubic}, and
\eqref{eq:subcase-A-diameter} gives the strict margin
\[
sD-C_g-C_k
\ge\frac7{120}-\frac1{S_0}+\alpha_k\rho_k
\ge\frac7{120}-\frac2{S_0^2}>0,
\]
since $S_0\ge6$.

\textbf{Subcase B: }The threads share an endpoint. By symmetry write
$k=ux$. If $d_G(x)=d\ge4$, \eqref{eq:adj-bound} gives
\[
 C_k\le
 \frac{d^2+11d-6}{2(d+3)^2(3d-1)}\le\frac7{120}.
\]
The last inequality follows after multiplication by the positive
denominator from
$f(d):=21d^3+59d^2-513d+297\ge0$ for $d\ge4$: indeed
$f(4)=533$ and $f'(d)=63d^2+118d-513>0$ on that range.  Thus
$C_g+C_k\le1/8=sH_s\le sD$.

Finally let $d=3$.  Since $D\ge H_s=3/4$, it suffices to prove
\begin{equation}\label{eq:cubic-correlation}
   6\rho_g+25\rho_k\ge17.
\end{equation}
Indeed, for $d=3$ one has $S_0=6$ and $\alpha_k=1/4$, so
\[
 C_g+C_k=\frac{2-\rho_g}{20}+\frac16-\frac5{24}\rho_k
 =\frac{32-6\rho_g-25\rho_k}{120},
\]
and, as $17\ge32-20D$, the displayed bound gives $C_g+C_k\le D/6=sD$.
Identify all vertices of $\widetilde B$ outside $\{u,v,w,x\}$ with one
node $o$, retaining the unit arms produced by this contraction. Add an arm
only for an incidence belonging to a deleted loop thread. Each exterior
incidence is therefore counted exactly once: the final multiplicity at
each of $u,v,x$ is its full degree minus its incidences inside the retained
subgraph. The resulting comparison networks are shown in
Fig.~\ref{fig:cubic-reduced}.  A direct series--parallel (equivalently,
Schur-complement) calculation, given in the matrix calculation below,
gives
\[
\begin{array}{c@{\qquad}c@{\qquad}c}
\toprule
&\bar\rho_g&\bar\rho_k\\
\midrule
vx\notin E(G)&22/31&17/31\\
vx\in E(G)&2/3&13/24\\
\bottomrule
\end{array}
\]
and Rayleigh monotonicity gives $\rho_g\ge\bar\rho_g$ and
$\rho_k\ge\bar\rho_k$.  The two resulting left sides of
\eqref{eq:cubic-correlation} are $557/31$ and $421/24$, both larger than $17$. The case in which $k$ shares $v$ is symmetric, and sharing both ends would contradict $e=0$ and simplicity.

\begin{figure}[ht]
\centering
\begin{tikzpicture}[scale=0.9]
  \node[circle,fill=black,inner sep=1.5pt] (w) at (0,1.5) {};
  \node[circle,fill=black,inner sep=1.5pt] (u) at (-2,0) {};
  \node[circle,fill=black,inner sep=1.5pt] (v) at (2,0) {};
  \node[circle,fill=black,inner sep=1.5pt] (x) at (0,0) {};
  \node[circle,fill=black,inner sep=1.5pt] (o) at (0,-1.5) {};
  \node[above] at (w) {$w$};
  \node[left] at (u) {$u$};
  \node[right] at (v) {$v$};
  \node[below left] at (x) {$x$};
  \node[below] at (o) {$o$};
  \draw (u)--(w) node[midway,left]{$1$};
  \draw (w)--(v) node[midway,right]{$1$};
  \draw (u)--(x) node[midway,below]{$1$};
  \draw (u)--(o) node[midway,left]{$1$};
  \draw (v)--(o) node[midway,right]{$2$};
  \draw (x)--(o) node[midway,right]{$2$};
  \node[below] at (0,-2.1) {(a) $vx\notin E(G)$};
\end{tikzpicture}
\qquad
\begin{tikzpicture}[scale=0.9]
  \node[circle,fill=black,inner sep=1.5pt] (w) at (0,1.5) {};
  \node[circle,fill=black,inner sep=1.5pt] (u) at (-2,0) {};
  \node[circle,fill=black,inner sep=1.5pt] (v) at (2,0) {};
  \node[circle,fill=black,inner sep=1.5pt] (x) at (0,0) {};
  \node[circle,fill=black,inner sep=1.5pt] (o) at (0,-1.5) {};
  \node[above] at (w) {$w$};
  \node[left] at (u) {$u$};
  \node[right] at (v) {$v$};
  \node[below left] at (x) {$x$};
  \node[below] at (o) {$o$};
  \draw (u)--(w) node[midway,left]{$1$};
  \draw (w)--(v) node[midway,right]{$1$};
  \draw (u)--(x) node[midway,below]{$1$};
  \draw (v)--(x) node[midway,above]{$1$};
  \draw (u)--(o) node[midway,left]{$1$};
  \draw (v)--(o) node[midway,right]{$1$};
  \draw (x)--(o) node[midway,right]{$1$};
  \node[below] at (0,-2.1) {(b) $vx\in E(G)$};
\end{tikzpicture}
\caption{The two reduced networks used in the cubic exceptional case.  All
depicted edges are unit edges, and the labels on the arms to $o$ record the
number of parallel unit arms at $u,v,x$, including the auxiliary arms for
deleted loop incidences. Each exterior incidence is counted once.}
\label{fig:cubic-reduced}
\end{figure}
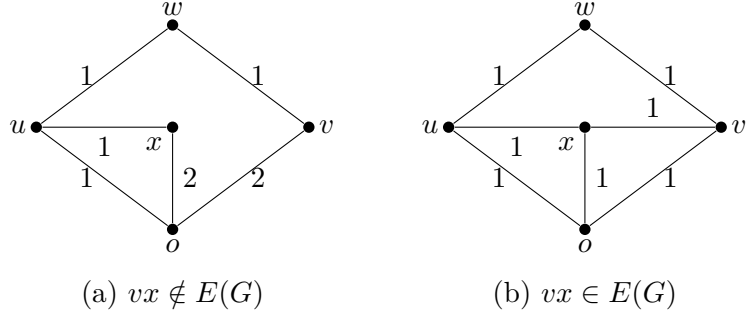

For $d\ge4$, $f(d)>0$ makes the local estimate strict. For $d=3$,
both displayed weighted sums exceed $17$. This proves strictness in
Subcase B as well.
\end{proof}

\subsection*{Grounded matrices}

For completeness, ground $o$ in the networks of
Fig.~\ref{fig:cubic-reduced} and order the remaining vertices as
$(u,v,w,x)$. If $vx\notin E(G)$, the reduced Laplacian is
\[
L_a=\begin{pmatrix}
3&0&-1&-1\\
0&3&-1&0\\
-1&-1&2&0\\
-1&0&0&3
\end{pmatrix}.
\]
In the second case, $vx\in E(G)$, each of $u,v,x$ has a single unit arm to
$o$, and the additional unit edge $vx$ is present; the reduced Laplacian is
\[
L_b=\begin{pmatrix}
3&0&-1&-1\\
0&3&-1&-1\\
-1&-1&2&0\\
-1&-1&0&3
\end{pmatrix}.
\]
Since $R(a,b)=(\ee_a-\ee_b)^{\mathsf T}L^{-1}(\ee_a-\ee_b)$ after grounding,
inverting these matrices gives
\[
\bar\rho_g=R(u,v)=\frac{22}{31},\qquad
\bar\rho_k=R(u,x)=\frac{17}{31}
\qquad(vx\notin E(G)),
\]
and
\[
\bar\rho_g=R(u,v)=\frac23,\qquad
\bar\rho_k=R(u,x)=\frac{13}{24}
\qquad(vx\in E(G)).
\]
These are the values used in the table above.

\end{document}